\documentclass{article}

\usepackage{amssymb,amsmath,amsthm}
\usepackage{cite}
\usepackage{hyperref}

\newtheorem{theorem}{Theorem}[section]
\newtheorem{lemma}[theorem]{Lemma}
\newtheorem{proposition}[theorem]{Proposition}

\newtheorem{definition}[theorem]{Definition}
\newtheorem{remark}[theorem]{Remark}

\numberwithin{equation}{section}

\allowdisplaybreaks

\begin{document}

\title{Generalized Fractional Operators on Time Scales
with Application to Dynamic Equations\thanks{Preprint whose 
final and definite form is with 
\href{https://link.springer.com/journal/11734}{\emph{The European Physical Journal Special Topics}} (EPJ ST), 
ISSN 1951-6355 (Print), ISSN 1951-6401 (Online).\newline
Submitted 20-June-2017; revised 26-Oct-2017; accepted for publication 06-April-2018.}}

\author{Kheira Mekhalfi$^1$\\
{\tt kheira.mekhalfi@yahoo.fr}
\and
Delfim F. M. Torres$^{2,}$\thanks{Corresponding author.}\\
{\tt delfim@ua.pt}}

\date{$^1$Institute of Sciences, Department of Mathematics,\\
University of Ain Temouchent, BP 284,\\
46000 Ain Temouchent, Algeria\\[0.3cm]
$^2$Center for Research and Development in Mathematics
and Applications (CIDMA), Department of Mathematics,\\
University of Aveiro, 3810-193 Aveiro, Portugal}

\maketitle


\begin{abstract}
We introduce more general concepts of Riemann--Liouville
fractional integral and derivative on time scales, of a
function with respect to another function.
Sufficient conditions for existence and uniqueness
of solution to an initial value problem described by
generalized fractional order differential equations
on time scales are proved.

\medskip


\medskip

\end{abstract}


\section{Introduction}

The theory of fractional differential equations, specifically the
question of existence and uniqueness of solutions, is a research
topic of great importance \cite{ABN,E.Bajlekova,E.Hernandez}.
Another important area of study is dynamic equations on time scales,
which goes back to 1988 and the work of Aulbach and Hilger,
and has been used with success to unify differential
and difference equations \cite{ABRP,MR1062633,AgBoh}.

Starting with a linear dynamic equation, Bastos et al.
have introduced the notion of fractional-order derivative
on time scales, involving time-scale analogues of
Riemann--Liouville operators \cite{BastosPhD,MR2728463,MyID:179}.
Another approach originate from the inverse Laplace
transform on time scales \cite{MR2800417}. After such pioneer work,
the study of fractional calculus on time scales developed
in a popular research subject: see
\cite{MyID:330,MyID:296,MyID:320,MyID:324,MyID:358}
and references therein. Recent results, since 2015,
cover fractional $q$-symmetric
systems on time scales \cite{MR3662497};
existence of solutions for delta Riemann--Liouville
fractional differential equations \cite{MR3537190};
existence and uniqueness of solutions for boundary-value problems
of fractional-order dynamic equations on time scales in Caputo sense
\cite{MR3498152}; existence of solutions for impulsive fractional
dynamic equations with delay on time scales \cite{MR3349274};
and existence of solutions for Cauchy problems with Caputo nabla
fractional derivatives \cite{MR3339669}. As a real application, we mention
the study of calcium ion channels that are retarded with injection
of calcium-chelator Ethylene Glycol Tetraacetic Acid \cite{MR3092034}.
In fact, physical applications
of fractional initial value problems in different
time scales abound \cite{MR1890104,Editorial}.
For example, fractional differential equations that
govern the behaviors of viscoelastic materials with memory
and the creep phenomenon have been proposed in \cite{MyID:350}
for the continuous time scale $\mathbb{T} = \mathbb{R}$;
fractional difference equations in discrete time scales
$\mathbb{T} = h\mathbb{Z}$, $h > 0$, or $\mathbb{T} = q^\mathbb{Z}$,
with relevance in the description of several physical phenomena,
are studied in \cite{MR2820857,MR3290224};
nonlocal thermally sensitive resistors on arbitrary time scales
$\mathbb{T}$ are investigated in \cite{MyID:365}.
Here, we extend available results in the literature
by introducing more general concepts of
fractional operators on time scales of a
function with respect to another function.
Then, we investigate corresponding generalized
fractional order dynamic equations on time scales.

The paper is organized as follows.
In Section~\ref{sec2}, we briefly recall
necessary definitions and results.
Our own results are then given in Section~\ref{sec3}:
we define a fractional integral operator
of Riemann--Liouville type on time scales
(Definition~\ref{def:GFI})
and a generalized fractional derivative
(Definition~\ref{def:GFD}).
Such generalizations on time scales
help us to study relations between
fractional difference equations and fractional differential equations.
On the other hand, they also provide a background to study boundary
value problems including fractional order difference and
differential equations. In this direction, we generalize
the recent results of \cite{MyID:328} by proving some
sufficient conditions for uniqueness and existence
of solutions of a fractional initial value problem
on an arbitrary time scale (see Theorems~\ref{th2.3} and \ref{2ndMR}).
We end with Section~\ref{sec4} of illustrative examples.


\section{Preliminaries}
\label{sec2}

In this section, we recall some definitions and results that are
used in the sequel. We use $\mathcal{C}(\mathcal{J},\mathbb{R})$
to denote the Banach space of continuous functions with the norm
$$
\|y\|_{\infty} = \sup\left\{|y(t)|: t\in \mathcal{J}\right\},
$$
where $\mathcal{J}$ is an interval.

A time scale $\mathbb{T}$ is an arbitrary nonempty closed subset
of the real numbers. The calculus on time scales was initiated
by Aulbach and Hilger \cite{AuHilger,S.Hilger} in order to create
a theory that can unify and extend discrete and continuous analysis.
The real numbers $\mathbb{R}$, the integers $\mathbb{Z}$, the natural
numbers $\mathbb{N}$, the nonnegative integers $\mathbb{N}_{0}$,
the $h$-numbers ($h\mathbb{Z} = \{hk : k \in \mathbb{Z}\}$,
where $h >0$ is a fixed real number), and the $q$-numbers
($q^{\mathbb{Z}} \cup \{0\} = \{q^{k} :k \in \mathbb{Z}\}\cup\{0\}$,
where $q > 1$ is a fixed real number), are examples of time scales,
as are $[0, 1]\cup[2, 3]$, $[0, 1]\cup\mathbb{N}$, and the Cantor set,
where $[0, 1]$ and $[2, 3]$ are intervals of real numbers.
Any time scale $\mathbb{T}$, being a closed subset of the real numbers,
has a topology inherited from the real numbers with the standard topology.
It is a complete metric space with the metric (distance)
$d : \mathbb{T}^2 \rightarrow \mathbb{R}$, $d(t, s) = |t - s|$
for $t, s \in \mathbb{T}$. Consequently, according
to the well-known theory of general metric spaces, we have
for $\mathbb{T}$ the fundamental concepts such as open balls (intervals),
neighborhoods of points, open sets, closed sets, compact sets, and so on.
In particular, for a given number $N > 0$, the $N$-neighborhood
$U_{\delta}(t)$ of a given point $t \in \mathbb{T}$ is the set of all points
$s \in \mathbb{T}$ such that $d(t, s) < N$. By a neighborhood of a point
$t  \in \mathbb{T}$ it is meant an arbitrary set in $\mathbb{T}$ containing
a $N$-neighborhood of the point $t$. Also, we have for functions
$f : \mathbb{T} \rightarrow \mathbb{R}$ the concepts of limit, continuity,
and the properties of continuous functions on general complete metric spaces
(note that, in particular, any function $f :  \mathbb{Z} \rightarrow \mathbb{R}$
is continuous at each point of $\mathbb{Z}$). The main task is to introduce
and investigate the concept of derivative for functions
$f :  \mathbb{T} \rightarrow \mathbb{R}$. This proves to be possible
due to the special structure of the metric space $\mathbb{T}$. In the definition
of derivative, an important role is played by the so-called forward and
backward jump operators.

\begin{definition}[See \cite{AgBoh}]
\label{def:sigma:rho}
For $t \in \mathbb{T}$, one defines the forward jump operator
$\sigma : \mathbb{T} \rightarrow \mathbb{T}$
by
\begin{equation*}
\sigma(t) = \inf\{s \in \mathbb{T} : s > t\},
\end{equation*}
while the backward jump operator
$\rho : \mathbb{T} \rightarrow \mathbb{T}$
is defined by
\begin{equation*}
\rho(t) = \sup\{s \in \mathbb{T} : s < t\}.
\end{equation*}
In addition, we put $\sigma(\max\mathbb{T}) = \max \mathbb{T}$
if there exists a finite $\max\mathbb{T}$,
and $\rho(\min\mathbb{T}) = \min\mathbb{T}$
if there exists a finite $\min\mathbb{T}$.
\end{definition}

Obviously, both $\sigma(t)$ and $\rho(t)$ are in $\mathbb{T}$
when $t \in \mathbb{T}$. This is because of our assumption
that $\mathbb{T}$ is a closed subset of $\mathbb{R}$.
Let $t \in \mathbb{T}$. If $\sigma(t) > t$, then we say that
$t$ is right-scattered, while if $\rho(t) < t$, then we say
that $t$ is left-scattered. Also, if $t < \max \mathbb{T}$
and $\sigma(t) = t$, then $t$ is called right-dense,
and if $t > \min\mathbb{T}$ and $\rho(t) = t$,
then $t$ is called left-dense.

The derivative makes use of the set $\mathbb{T}^{\kappa}$,
which is derived from the time scale $\mathbb{T}$ as follows:
if $\mathbb{T}$ has a left-scattered maximum $M$, then
$\mathbb{T}^{\kappa}:=\mathbb{T} \setminus \{M\}$;
otherwise, $\mathbb{T}^{\kappa}:=\mathbb{T}$.

\begin{definition}[Delta derivative \cite{AB}]
Assume $f:\mathbb{T}\rightarrow \mathbb{R}$ and let
$t\in \mathbb{T}^{\kappa}$. One defines
$$
f^{\Delta}(t):=\lim_{s\rightarrow t}\frac{f(\sigma(s))-f(t)}{\sigma(s)-t},
\quad t \neq \sigma(s),
$$
provided the limit exists. We call $f^{\Delta}(t)$ the delta derivative
(or Hilger derivative) of $f$ at $t$. Moreover, we say that $f$
is delta differentiable on $\mathbb{T}^{\kappa}$ provided
$f^{\Delta}(t)$ exists for all $t\in \mathbb{T}^{\kappa}$. The function
$f^{\Delta}:\mathbb{T}^{\kappa}\rightarrow \mathbb{R}$ is then called
the (delta) derivative of $f$ on $\mathbb{T}^{\kappa}$.
\end{definition}

\begin{definition}[Delta Integral \cite{AB,AgBoh}]
Let $[a,b]$ denote a closed bounded interval in $\mathbb{T}$.
A function $F: [a,b]\rightarrow \mathbb{R}$ is called a delta antiderivative
of function $f: [a,b)\rightarrow \mathbb{R}$ provided $F$ is continuous
on $[a,b]$, delta differentiable on $[a,b)$, and $ F^{\Delta}(t)=f(t)$
for all $t\in [a,b)$. Then, we define the $\Delta$-integral of $f$
from $a$ to $b$ by
\begin{equation*}
 \int_{a}^{b}f(t)\Delta t := F(b)-F(a).
\end{equation*}
\end{definition}

\begin{definition}[See \cite{AB,AgBoh}]
A function $f:\mathbb{T}\rightarrow \mathbb{R}$ is called rd-continuous,
provided it is continuous at right-dense points in $\mathbb{T}$
and its left-sided limits exist (finite) at left-dense points
in $\mathbb{T}$. The set of rd-continuous functions
$f:\mathbb{T}\rightarrow \mathbb{R}$ is denoted by $\mathcal{C}_{rd}$.
Similarly, a function $f:\mathbb{T}\rightarrow \mathbb{R}$
is called ld-continuous provided it is continuous at left-dense points
in $\mathbb{T} $ and its right-sided limits exist (finite)
at right-dense points in $\mathbb{T}$. The set of ld-continuous
functions $f:\mathbb{T}\rightarrow \mathbb{R}$ is denoted by $\mathcal{C}_{ld}$.
\end{definition}

All rd-continuous bounded functions on $[a, b)$ are delta integrable from
$a$ to $b$. For a more general treatment of the delta integral on time scales
(Riemann, Lebesgue, and Riemann-Stieltjes integration on time scales),
see \cite{AgBoh,Guseinov,MyID:137}.

\begin{proposition}[See \cite{AgBoh}]
\label{P1}
Suppose $a$, $b \in \mathbb{T}$, $a < b$,
and $f(t)$ is continuous on $[a, b]$. Then,
\begin{equation*}
\int_{a}^{b}f(t)\Delta t
= [\sigma(a) - a]f(a)
+\int_{\sigma(a)}^{b}f(t)\Delta t.
\end{equation*}
\end{proposition}

\begin{proposition}[See \cite{AJ}]
\label{P2}
Suppose $\mathbb{T}$ is a time scale and $f$ is an increasing continuous function
on  $[a,b]$. If $F$ is the extension of $f$ to the real interval $[a,b]$ given by
\begin{equation*}
F(s) :=
\begin{cases}
f(s) & \textrm{ if } s \in\mathbb{T} , \\
f(t) & \textrm{ if } s \in (t,\sigma(t))\not\subset\mathbb{T},
\end{cases}
\end{equation*}
then
\begin{equation*}
\int_{a}^{b} f(t) \Delta t \leq \int_{a}^{b} F(t)dt.
\end{equation*}
\end{proposition}

\begin{definition}[Fractional integral on time scales \cite{MyID:328}]
\label{def:FI}
Suppose $\mathbb{T}$ is a time scale, $[a,b]$ is an interval of $\mathbb{T}$,
and $h$ is an integrable function on $[a,b]$. Let $0 < \alpha <1$.
Then the (left) fractional integral of order $\alpha$ of $h$ is defined by
$$
{_{a}^{\mathbb{T}}I}_{t}^{\alpha}h(t)
:= \int_{a}^{t} \frac{(t-s)^{\alpha-1}}{\Gamma(\alpha)}h(s)\Delta s,
$$
where $\Gamma$ is the gamma function.
\end{definition}

Using Definition~\ref{def:FI}, one defines the Riemann--Liouville
fractional derivative on time scales.

\begin{definition}[Riemann--Liouville derivative on time scales \cite{MyID:328}]
\label{def:FD}
Let $\mathbb{T}$ be a time scale, $t\in\mathbb{T}$, $0 < \alpha <1$,
and $h:\mathbb{T}\rightarrow \mathbb{R}$. The (left) Riemann--Liouville
fractional derivative of order $\alpha$ of $h$ is defined by
\begin{equation*}
{_{a}^{\mathbb{T}}D}_{t}^{\alpha}h(t)
:=\frac{1}{\Gamma(1-\alpha)}\left(\int_{a}^{t}
(t-s)^{-\alpha}h(s)\Delta s\right)^{\Delta}.
\end{equation*}
\end{definition}


\section{Main Results}
\label{sec3}

We begin by generalizing the concepts of fractional integral
and fractional derivative given by Definitions~\ref{def:FI}
and \ref{def:FD}, by introducing the concept of fractional
integration and fractional differentiation on time scales
of a function with respect to another function.

\begin{definition}[Generalized fractional integral on time scales]
\label{def:GFI}
Suppose $\mathbb{T}$ is a time scale, $[a,b]$ is an interval of $\mathbb{T}$,
$h$ is an integrable function on $[a,b]$, and $g$ is monotone having
a delta derivative $g^{\Delta}$ with $g^{\Delta}(t)\neq0$ for any $t\in[a,b]$.
Let $0 < \alpha <1$. Then, the (left) generalized fractional integral
of order $\alpha$ of $h$ with respect to $g$ is defined by
$$
{_{a;g}^{\mathbb{T}}I}_{t}^{\alpha}h(t)
= \int_{a}^{t} \frac{1}{\Gamma(\alpha)}
(g(t)-g(s))^{\alpha-1}g^{\Delta}(s)h(s)\Delta s.
$$
\end{definition}

\begin{definition}[Generalized fractional derivative on time scales]
\label{def:GFD}
Suppose $\mathbb{T}$ is a time scale, $[a,b]$ is an interval of $\mathbb{T}$,
$h$ is an integrable function on $[a,b]$, and $g$ is monotone having
a delta derivative $g^{\Delta}$ with $g^{\Delta}(t)\neq0$ for any $t\in[a,b]$.
Let $0 < \alpha <1$. Then, the (left) generalized fractional derivative
of order $\alpha$ of $h$ with respect to $g$ is defined by
$$
{_{a;g}^{\mathbb{T}}D}_{t}^{\alpha}h(t)
= \frac{1}{\Gamma(1-\alpha)}\frac{1}{g^{\Delta}(t)}\left(
\int_{a}^{t} (g(t)-g(s))^{-\alpha}g^{\Delta}(s)h(s)\Delta s\right)^{\Delta}.
$$
\end{definition}

\begin{remark}
If $\mathbb{T}=\mathbb{R}$, then Definitions~\ref{def:GFI} and \ref{def:GFD}
give, respectively, the well-known generalized fractional integral and derivative
of Riemann--Liouville \cite[Section~18.2]{SaKiMa}.
\end{remark}

Let $z$ be a monotone function having a delta derivative
$z^{\Delta}$ with $z^{\Delta}(t)\neq0$ for any $t\in \mathcal{J}$.
We consider the following initial value problem:
\begin{equation}
\label{eq1}
\begin{gathered}
_{t_{0};z}^{\mathbb{T}}D^{\alpha}_{t}y(t)=f(t,y(t)),
\quad  t\in [t_{0},t_{0}+a]
=\mathcal{J}\subseteq \mathbb{T},
\quad 0<\alpha<1,\\
_{t_{0};z}^{\mathbb{T}}I^{\alpha}_{t}y(t_{0})=0,
\end{gathered}
\end{equation}
where $_{t_{0};z}^{\mathbb{T}}D^{\alpha}_{t}$
and $_{t_{0};z}^{\mathbb{T}}I^{1-\alpha}_{t}$
is the (left) Riemann--Liouville generalized fractional
derivative and integral with respect to function $z$, and
$f : \mathcal{J} \times\mathbb{R} \rightarrow \mathbb{R}$
is a right-dense continuous function. Our main goal is to obtain
sufficient conditions for the existence and uniqueness
of solution to problem \eqref{eq1}.

\medskip

In what follows, $\mathbb{T}$ is a given time scale and
$\mathcal{J}=[t_{0},t_{0}+a] \subseteq \mathbb{T}$.

\begin{lemma}
\label{Lem:sol}
Let $0<\alpha< 1$, $\mathcal{J}\subseteq\mathbb{T}$, and
$f: \mathcal{J}\times\mathbb{R}\rightarrow\mathbb{R}$.
Function $y\in\mathcal{C}(\mathcal{J},\mathbb{R})$
is a solution of problem \eqref{eq1}
if and only if is a solution
of the following integral equation:
$$
y(t)=\frac{z^{\Delta}(t)}{\Gamma(\alpha)}\int_{t_{0}}^{t}
(z(t)-z(s))^{\alpha-1}z^{\Delta}(s)f(s,y(s))\Delta s.
$$
\end{lemma}

\begin{proof}
By Definitions~\ref{def:GFI} and \ref{def:GFD}, we have
\begin{equation*}
{_{t_{0};z}^{\mathbb{T}}I}_{t}^{1-\alpha}y(t)
=   \frac{1}{\Gamma(1-\alpha)}\int_{t_{0}}^{t}(z(t)-z(s))^{
-\alpha}z^{\Delta}(s)y(s)\Delta s
\end{equation*}
and 
\begin{equation*}
\begin{split}
{_{t_{0};z}^{\mathbb{T}}D}_{t}^{\alpha}y(t)
&= \frac{1}{z^{\Delta}(t)}\frac{1}{\Gamma(1-\alpha)}\left(\int_{t_{0}}^{t} (z(t)-z(s))^{-\alpha}z^{\Delta}(s)y(s)\Delta s\right)^{\Delta}\\
&= \frac{1}{z^{\Delta}(t)}\left({_{t_{0};z}^{\mathbb{T}}I}_{t}^{1-\alpha}y(t)\right)^{\Delta}
=\frac{1}{z^{\Delta}(t)}\left(\Delta \circ{_{t_{0};z}^{\mathbb{T}}I}_{t}^{1-\alpha}\right)y(t).
\end{split}
\end{equation*}
Moreover,
\begin{equation}
\label{eq:proof:lem3.4}
\begin{split}
{_{t_{0};z}^{\mathbb{T}}I}_{t}^{\alpha}[{_{t_{0};z}^{\mathbb{T}}D}_{t}^{\alpha}y(t)]
&=\frac{1}{z^{\Delta}(t)} \left({_{t_{0};z}^{\mathbb{T}}I}_{t}^{\alpha}(
{_{t_{0};z}^{\mathbb{T}}I}_{t}^{1-\alpha}y(t))\right)^{\Delta}\\
&= \frac{1}{z^{\Delta}(t)}\left({_{t_{0};z}^{\mathbb{T}}I}_{t}^{1}y(t)\right)^{\Delta}\\
&=  \frac{1}{z^{\Delta}(t)}y(t).
\end{split}
\end{equation}
It follows from \eqref{eq:proof:lem3.4}, \eqref{eq1}
and Definition~\ref{def:GFI} that
\begin{equation*}
\begin{split}
y(t) &= z^{\Delta}(t) {_{t_{0};z}^{\mathbb{T}}I}_{t}^{\alpha}\left[
{_{t_{0};z}^{\mathbb{T}}D}_{t}^{\alpha}y(t)\right]\\
&= z^{\Delta}(t){_{t_{0};z}^{\mathbb{T}}I}_{t}^{\alpha}f(t,y(t))\\
&= \frac{z^{\Delta}(t)}{\Gamma(\alpha)}\int_{t_{0}}^{t}
(z(t)-z(s))^{\alpha-1}z^{\Delta}(s)f(s,y(s))\Delta s.
\end{split}
\end{equation*}
The proof is complete.
\end{proof}

Our first main result is based on the
Banach fixed point theorem \cite{A.Granas}.

\begin{theorem}
\label{th2.3}
Let $f:\mathcal{J}\times \mathbb{R} \rightarrow \mathbb{R}$
be continuous and assume there exists a constant $L>0$ such that
\begin{equation*}
|f(t,u)-f(t,v)| \leq  L \|u-v\|_{\infty}
\end{equation*}
for $t\in  \mathcal{J}$ and
$u,v \in \mathcal{C}(\mathcal{J},\mathbb{R})$.
If
\begin{equation}
\label{eq1.4a}
\frac{Lz^{\Delta}(t) \, M_{\alpha}(t)} \, (t-t_{0}){\Gamma(\alpha)}<1,
\quad t\in  \mathcal{J},
\end{equation}
where
\begin{equation}
\label{eq:M:alpha}
M_{\alpha}(t) := \frac{\displaystyle
\int_{t_{0}}^{t}(z(t)-z(s))^{\alpha-1}z^{\Delta}(s)ds}{t-t_{0}},
\end{equation}
then problem \eqref{eq1} has a unique solution
on $\mathcal{J}$.
\end{theorem}

\begin{proof}
We transform problem \eqref{eq1} into a fixed point problem.
Consider the operator
$F : \mathcal{C}(\mathcal{J},\mathbb{R})\to \mathcal{C}(\mathcal{J},\mathbb{R})$
defined by
\begin{equation}
\label{eq3:b}
F(y)(t)=\frac{z^{\Delta}(t)}{\Gamma(\alpha)}
\int_{t_{0}}^{t}(z(t)-z(s))^{\alpha-1}z^{\Delta}(s) f(s,y(s)) \Delta s.
\end{equation}
We need to prove that $F$ has a fixed point, which is a
unique solution of \eqref{eq1} on $\mathcal{J}$.
For that, we show that $F$ is a contraction.
Let $x,y \in \mathcal{C}(\mathcal{J},\mathbb{R})$.
For $t\in \mathcal{J}$, we have
\begin{equation*}
\begin{split}
|F(x)(t)&-F(y)(t)|\\
&= \left|  \frac{z^{\Delta}(t)}{\Gamma(\alpha)}
\int_{t_{0}}^{t}(z(t)-z(s))^{\alpha-1}z^{\Delta}(s) [f(s,x(s))-f(s,y(s)) ]\Delta s\right|\\
&\leq \frac{z^{\Delta}(t)}{\Gamma(\alpha)}\int_{t_{0}}^{t}(z(t)-z(s))^{\alpha-1}
z^{\Delta}(s) |f(s,x(s))-f(s,y(s))|\Delta s\\
&\leq \frac{L z^{\Delta}(t)\|x-y\|_{\infty}}{\Gamma(\alpha)}\int_{t_{0}}^{t}
(z(t)-z(s))^{\alpha-1}z^{\Delta}(s)\Delta s\\
&\leq \frac{L z^{\Delta}(t)\|x-y\|_{\infty}}{\Gamma(\alpha)}
\int_{t_{0}}^{t}(z(t)-z(s))^{\alpha-1}z^{\Delta}(s)ds.
\end{split}
\end{equation*}
Then, it follows from \eqref{eq:M:alpha} that
\begin{equation*}
|F(x)(t)-F(y)(t)|
\leq \frac{L z^{\Delta}(t) M_{\alpha}(t) (t-t_{0})}{\Gamma(\alpha)}\| x-y\|_{\infty}.
\end{equation*}
By \eqref{eq1.4a}, $F$ is a contraction and thus, by the contraction mapping theorem,
we deduce that $F$ has a unique fixed point. This fixed point is the unique solution
of \eqref{eq1}.
\end{proof}

Now, we give our second main result, which is an existence result
based upon the nonlinear alternative of Leray--Schauder,
applied to completely continuous operators \cite{A.Granas}.

\begin{theorem}
\label{2ndMR}
Suppose $f :\mathcal{J}\times\mathbb{R}\rightarrow\mathbb{R}$
is a rd-continuous bounded function such that there exists
$N > 0$ with $|f(t,y)|\leq N$ for all $t\in\mathcal{J}$, $y\in\mathbb{R}$.
Then problem \eqref{eq1} has a solution on $\mathcal{J}$.
\end{theorem}

\begin{proof}
We use Schauder's fixed point theorem \cite{A.Granas}
to prove that $\mathrm{F}$ defined by \eqref{eq3:b}
has a fixed point. The proof is given in four steps.\\
\emph{Step 1:} $\mathrm{F}$ is continuous.
Let $y_{n}$ be a sequence such that $y_{n}\rightarrow y$
in $\mathcal{C}(\mathcal{J},\mathbb{R})$.
Then, for each $t\in\mathcal{J}$,
\begin{eqnarray*}
|F(y_{n})(t)
&-&F(y)(t)|\\
&\leq&\frac{z^{\Delta}(t)}{\Gamma(\alpha)}
\int_{t_{0}}^{t}(z(t)-z(s))^{\alpha-1}z^{\Delta}(s)
\left|f(s,y_{n}(s))-f(s,y(s))\right|\Delta s\\
&\leq&\frac{z^{\Delta}(t)}{\Gamma(\alpha)}
\int_{t_{0}}^{t}(z(t)-z(s))^{\alpha-1}z^{\Delta}(s)
\sup_{s\in\mathcal{J}}\left|f(s,y_{n}(s))-f(s,y(s))\right|\Delta s\\
&= & \frac{z^{\Delta}(t)\left \|f(\cdot,y_{n}(\cdot))
-f(\cdot,y(\cdot))\right\|_{\infty}}{\Gamma(\alpha)}
\int_{t_{0}}^{t}(z(t)-z(s))^{\alpha-1}z^{\Delta}(s)\Delta s\\
&\leq& \frac{z^{\Delta}(t)\|f(\cdot,y_{n}(\cdot))
-f(\cdot,y(\cdot))\|_{\infty}}{\Gamma(\alpha)}
\int_{t_{0}}^{t}(z(t)-z(s))^{\alpha-1}z^{\Delta}(s)d s\\
&\leq&\frac{z^{\Delta}(t)M_{\alpha}(t) (t-t_{0})}{
\Gamma(\alpha)}\left\|f(\cdot,y_{n}(\cdot))
-f(\cdot,y(\cdot))\right\|_{\infty}.
\end{eqnarray*}
Since $f$ is a continuous function, we have
\begin{multline*}
\left|F(y_{n})(t)-F(y)(t)\right|\\
\leq\frac{z^{\Delta}(t)M_{\alpha}(t) (t-t_{0})}{
\Gamma(\alpha)}\left\|f(\cdot,y_{n}(\cdot))
-f(\cdot,y(\cdot))\right\|_{\infty}
\longrightarrow 0 \ \text{ as } \ n\rightarrow \infty.
\end{multline*}
\emph{Step 2:} the map $F$ sends bounded sets into bounded sets
in $\mathcal{C}(\mathcal{J},\mathbb{R})$. 
It is enough to show that there exists 
a positive constant $l$ such that
$$
F(y)\in B_{l}=\{F(y)\in\mathcal{C}(\mathcal{J},\mathbb{R})
: \|\mathrm{F}(y)\|_{\infty}\leq l \}.
$$
By hypothesis, for each $t\in\mathcal{J}$, one has
\begin{equation*}
\begin{split}
|\mathrm{F}(y)(t)|&\leq\frac{z^{\Delta}(t)}{\Gamma(\alpha)}
\int_{t_{0}}^{t}(z(t)-z(s))^{\alpha-1}z^{\Delta}(s)|f(s,y(s))|\Delta s\\
&\leq\frac{z^{\Delta}(t)N}{\Gamma(\alpha)}
\int_{t_{0}}^{t} (z(t)-z(s))^{\alpha-1}z^{\Delta}(s)\Delta s\\
&\leq\frac{z^{\Delta}(t)N}{\Gamma(\alpha)}
\int_{t_{0}}^{t}(z(t)-z(s))^{\alpha-1}z^{\Delta}(s)ds\\
&\leq\frac{z^{\Delta}(t)N M_{\alpha}(t) (t-t_{0})}{\Gamma(\alpha)}=l.
\end{split}
\end{equation*}
\emph{Step 3:} $F$ sends bounded sets into equicontinuous sets
of $\mathcal{C}(\mathcal{J},\mathbb{R})$. 
Let  $t_{1}, t_{2} \in \mathcal{J}$, $t_{1} < t_{2}$. Then,
\begin{eqnarray*}
|F(y)(t_{2})&-&F(y)(t_{1})|\\
&\leq&\frac{z^{\Delta}(t)}{\Gamma(\alpha)}\left|\int_{t_{0}}^{t_{1}}
(z(t_{1})-z(s))^{\alpha-1}z^{\Delta}(s)f(s,y(s))\Delta s\right.\\
&&\left. \qquad -\int_{t_{0}}^{t_{2}}(z(t_{2})-z(s))^{\alpha-1}
z^{\Delta}(s)f(s,y(s))\Delta s\right|\\
&\leq&\frac{z^{\Delta}(t)}{\Gamma(\alpha)}\left|
\int_{t_{0}}^{t_{1}}\left((z(t_{1})-z(s))^{\alpha-1}
-(z(t_{2})-z(s))^{\alpha-1}\right.\right.\\
&& \qquad \left.\left. +(z(t_{2})-z(s))^{\alpha-1}\right)z^{\Delta}(s)f(s,y(s))\Delta s\right.\\
&&\qquad \left.-\int_{t_{0}}^{t_{2}}(z(t_{2})-z(s))^{\alpha-1}z^{\Delta}(s)
f(s,y(s))\Delta s\right|\\
&\leq&\frac{z^{\Delta}(t)N}{\Gamma(\alpha)}
\left|\int_{t_{0}}^{t_{1}}((z(t_{1})-z(s))^{\alpha-1}
-(z(t_{2})-z(s))^{\alpha-1})z^{\Delta}(s)\Delta s\right.\\
&&\qquad \left. +\int_{t_{1}}^{t_{2}}(z(t_{2})-z(s))^{\alpha-1}z^{\Delta}(s)
\Delta s\right|\\
&\leq&\frac{z^{\Delta}(t)N}{\Gamma(\alpha)}\left|
\int_{t_{0}}^{t_{1}}((z(t_{1})-z(s))^{\alpha-1}
-(z(t_{2})-z(s))^{\alpha-1})z^{\Delta}(s)ds\right.\\
&&\qquad \left. +\int_{t_{1}}^{t_{2}}(z(t_{2})-z(s))^{\alpha-1}
z^{\Delta}(s) d s\right|.
\end{eqnarray*}
As $t_{1}\rightarrow t_{2}$, the right-hand side of the above inequality
tends to zero. As a consequence of Steps~1 to 3, together with the Arzela--Ascoli
theorem, we conclude that $F : \mathcal{C}(\mathcal{J},\mathbb{R})
\rightarrow \mathcal{C}(\mathcal{J},\mathbb{R})$ is completely continuous.\\
\emph{Step 4:} a priori boundedness of solutions. It remains to show that the set
\begin{equation*}
\mathcal{E}=\{y\in \mathcal{C}(\mathcal{J},\mathbb{R})
: y=\lambda F(y),~~ 0<\lambda<1\}
\end{equation*}
is bounded. Let $y\in \mathcal{E}$ be any element.
Then, for each $t\in \mathcal{J}$,
\begin{equation*}
y(t)=\lambda F(y)(t)
=\lambda \frac{z^{\Delta}(t)}{\Gamma(\alpha)}
\int_{t_{0}}^{t}(z(t)-z(s))^{\alpha-1}z^{\Delta}(s)f(s,y(s))\Delta s.
\end{equation*}
It follows that
\begin{eqnarray*}
|y(t)|
&\leq& \bigg|\frac{z^{\Delta}(t)}{\Gamma(\alpha)}
\int_{t_{0}}^{t}(z(t)-z(s))^{\alpha-1}z^{\Delta}(s)f(s,y(s))\Delta s\bigg|\\
&\leq&\frac{z^{\Delta}(t)}{\Gamma(\alpha)}
\int_{t_{0}}^{t}(z(t)-z(s))^{\alpha-1}z^{\Delta}(s)|f(s,y(s))|\Delta s\\
&\leq&\frac{z^{\Delta}(t)N}{\Gamma(\alpha)}
\int_{t_{0}}^{t}(z(t)-z(s))^{\alpha-1}z^{\Delta}(s)\Delta s\\
&\leq& \frac{z^{\Delta}(t)N}{\Gamma(\alpha)}
\int_{t_{0}}^{t}(z(t)-z(s))^{\alpha-1}z^{\Delta}(s)d s\\
&\leq& \frac{z^{\Delta}(t)N M_{\alpha}(t) (t-t_{0})}{\Gamma(\alpha)}.
\end{eqnarray*}
Hence, the set $\mathcal{E}$ is bounded. As a consequence
of Schauder's fixed point theorem, we conclude that $F$
has a fixed point, which is solution of \eqref{eq1}.
\end{proof}


\section{Examples}
\label{sec4}

When function $z$ is the identity $id$, we have
$I_{a;id}^{\alpha}=I_{a}^{\alpha}$, that is,
we get the ordinary left Riemann--Liouville fractional integral.
In this very particular case, existence of solution
to fractional problem \eqref{eq1} on time scales
has been recently investigated in \cite{MyID:328}.
Here we cover the general situation: existence of solution
to the fractional initial value problem \eqref{eq1}, for a general function $z$.
For illustrative purposes, let $z(t)=t^{2}$, $\mathbb{T}=2^{\mathbb{N}}$,
$\alpha=\frac{1}{2}$, and $t_{0}=0$. Then, $\sigma(t)=2t$ and
\begin{equation*}
\begin{split}
(I_{0;t^{2}}^{\alpha}f)(t)
&=  \frac{(t^{2})^{\Delta}}{\Gamma(\alpha)}
\int_{0}^{t}(t^{2}-s^{2})^{\alpha-1}(s^{2})^{\Delta}f(s)\Delta s\\
&= \frac{(t^{2})^{\Delta}}{\sqrt{\pi}}
\int_{0}^{t}(t^{2}-s^{2})^{-\frac{1}{2}}(s^{2})^{\Delta}f(s)\Delta s.
\end{split}
\end{equation*}
In this case,
\begin{equation*}
F(y)(t)=\frac{(t^{2})^{\Delta}}{\sqrt{\pi}}
\int_{0}^{t}(t^{2}-s^{2})^{-\frac{1}{2}}(s^{2})^{\Delta}f(s,y(s))\Delta s
\end{equation*}
and we have
\begin{eqnarray*}
|F(x)(t)-F(y)(t)|
&\leq& \frac{(t+\sigma(t))L\|x-y\|_{\infty}}{\sqrt{\pi}}
\int_{0}^{t}(t^{2}-s^{2})^{-\frac{1}{2}}(s+\sigma(s)) \Delta s\\
&=&\frac{(t+2t)L\|x-y\|_{\infty}}{\sqrt{\pi}}
\int_{0}^{t}(t^{2}-s^{2})^{-\frac{1}{2}}(s+2s) \Delta s\\
&\leq& \frac{3tL\|x-y\|_{\infty}}{\sqrt{\pi}}
\int_{0}^{t}(t^{2}-s^{2})^{-\frac{1}{2}}(3s) ds\\
&\leq& \frac{9t^{2}L}{\sqrt{\pi}}\|x-y\|_{\infty}.
\end{eqnarray*}
In this example, 
$$
M_{\alpha}(t)=\frac{\displaystyle \int_{0}^{t}(t^{2}-s^{2})^{-\frac{1}{2}}(3s) ds}{t}
$$ 
and \eqref{eq:M:alpha} is reduced to $M_{\alpha}(t) \equiv 3$.
Choose $b$ such that
$$
b = \frac{9t^{2}L}{\sqrt{\pi}}<1.
$$
Then the conditions of Theorem~\ref{th2.3} are satisfied,
and we conclude that there is a function 
$y\in C(\mathcal{J},\mathbb{R})$
solution of \eqref{eq1}.


\section*{Acknowledgments}

This research was initiated
while Mekhalfi was visiting the Department
of Mathematics of University of Aveiro, Portugal, February 2017.
The hospitality of the host institution and the financial support
of University of Ain Temouchent, Algeria,
are here gratefully acknowledged. Torres was supported
by Portuguese funds through CIDMA and FCT, within project
UID/MAT/04106/2013. The authors are very grateful to two 
anonymous referees, for their suggestions and invaluable comments.


\section*{Author Contribution Statement}

Both authors contributed equally to the paper.



\end{document}